\documentclass[11pt,bezier]{article}
\pdfoutput=1
\usepackage{amsmath}
\usepackage{amsthm, thm-restate}
\usepackage{amsfonts,amsthm,amssymb}
\usepackage{enumitem}
\usepackage{amsfonts}
\usepackage{graphics}
\usepackage{graphicx}
\usepackage{epstopdf}
\usepackage{subfigure}
\usepackage{cite}
\usepackage{amssymb}
\usepackage{mathtools}

\newtheorem{lemma}{Lemma}[section]
\newtheorem{corollary}{Corollary}[section]

\newtheorem{conjecture}{Conjecture}[section]

\newtheorem*{remark*}{Remark}
\theoremstyle{definition}

\begin{document}
	\title{Hamiltonian cycles and Hamiltonian paths in $2k$-connected, $1$-tough and $(P_{3}\cup kP_{1})$-free graphs\footnote{The research is supported by Natural Science Foundation of Xinjiang Uygur Autonomous Region (2025D01E02) and National Natural Science Foundation of China (12261086).}}
	\author{Hui Liu, Yingzhi Tian\footnote{Corresponding author. E-mail: huiliu2192000@163.com (H. Liu), tianyzhxj@163.com (Y. Tian).} \\
		{\small College of Mathematics and System Sciences, Xinjiang
			University, Urumqi, Xinjiang, 830046, PR China}}

	\date{}
	
	\maketitle
	
	\noindent{\bf Abstract } A graph $G$ is called Hamiltonian if it possesses a Hamiltonian cycle; and $G$ is called Hamiltonian-connected if it contains a Hamiltonian path between any two distinct vertices. The toughness of a non-complete graph is the minimum ratio of $|S|$ to the number of components of $G-S$ for any cutset $S$. For a given graph $H$, a graph $G$ is called $H$-free if $G$ does not contain $H$ as an induced subgraph. In this paper, for an integer $k\ge 2$, we prove that every $2k$-connected, $1$-tough and $(P_{3}\cup kP_{1})$-free graph is Hamiltonian and every $(2k+1)$-connected $(P_{3}\cup kP_{1})$-free graph with toughness greater than $1$ is Hamiltonian-connected.
	

	\section{Introduction}
	The interconnection network of a parallel or distributed system is commonly modeled as an undirected graph $G=(V(G),E(G))$. A primary goal in network design is to ensure efficient communication. Hamiltonian cycles and Hamiltonian paths are among the most desirable topological properties in such networks, as they provide efficient communication schemes and fault-tolerant strategies. A Hamiltonian cycle in a graph $G$ is a cycle that contains every vertex of $G$. A graph $G$ is called Hamiltonian if it admits a Hamiltonian cycle. A Hamiltonian path in $G$ is a path that traverses every vertex of $G$ exactly once. A graph $G$ is called Hamiltonian-connected if for every pair of distinct vertices $u,v\in V(G)$, there exists a Hamiltonian path with endpoints $u$ and $v$. Hamiltonian properties are of fundamental importance in graph theory. As classic structural properties, they link local conditions (such as degree, connectivity, or toughness) to global traversability, thereby revealing deep insights into the architecture of graphs. This significance extends naturally to many practical domains. In network design, a Hamiltonian cycle provides a ring topology that supports efficient broadcasting and also guarantees a simple, fault-tolerant routing scheme where data can be transmitted along a cycle with built-in redundancy. In algorithmic graph theory, the Hamiltonian cycle problem is a classic NP-complete problem; the search for polynomial-time sufficient conditions has driven research in combinatorial optimization, approximation algorithms, and parameterized complexity for decades. In operations research, the traveling salesman problem, which asks for a minimum weighted Hamiltonian cycle in a weighted complete graph, has direct applications in logistics, delivery route planning, and supply chain management. In electronic design automation, Hamiltonian paths are used to optimize the drilling order on printed circuit boards. These diverse applications explain why Hamiltonian properties remain a vibrant research topic. In recent years, research on Hamiltonian graphs and Hamiltonian-connected graphs has progressed along several major directions in graph theory. One active direction explores the connection between toughness and spectral radius of a graph\cite{Chen,Liu,LiuF,Lou}. Another important direction is the classification of forbidden induced subgraphs that force Hamiltonian properties\cite{Broersma and Patel,Broersma,Gao,Li,Matthews,Ota and Sanka,Ota and SankaP1,Shan,Shan p2p3,Xu}. A third line of inquiry focuses on degree sequence conditions for Hamiltonian properties in tough graphs\cite{Hoang,HoangRobin,ShanTan}. In this paper, given a forbidden induced subgraph condition together with assumptions on toughness and connectivity, we investigate the existence of Hamiltonian cycles and Hamiltonian paths in such graphs.
	
	The connectivity $\kappa(G)$ of a graph $G$ is the minimum size of a vertex set $S\subseteq V(G)$ such that $G-S$ is disconnected or has only one vertex. A graph is $k$-connected if $\kappa(G)\ge k$. For a non-complete graph $G$, the toughness $\tau(G)$ is defined as
	\[
	\tau(G)=\min\left\{\frac{|S|}{c(G-S)}: S\subseteq V(G),\; c(G-S)\ge 2\right\},
	\]
	where $c(G-S)$ denotes the number of components of $G-S$. For complete graphs we set $\tau(G)=\infty$. A graph is called $t$-tough if $\tau(G)\ge t$. It is easy to see that every Hamiltonian graph is $1$-tough, and for any non-complete graph, $\kappa(G)\ge 2\tau(G)$.
	
		In 1973, Chv\'{a}tal \cite{Chvatal} proposed the following conjecture.
	\begin{conjecture}(Chv\'{a}tal \cite{Chvatal})
		There exists a constant $t_{0}>0$ such that every $t_{0}$-tough graph on at least three vertices is Hamiltonian.
	\end{conjecture}
	
	Bauer, Broersma and Veldman\cite{Bauer} proved that for arbitrary $\epsilon >0$, there exists a $(\frac{9}{4}-\epsilon )$-tough graph without a Hamiltonian path. Thus, if Conjecture $1.1$ is true, then such a constant $t_{0}$ must be at least $\frac{9}{4}$. Nevertheless, despite the uncertainty about the exact value of $t_0$, Conjecture $1.1$ has been proved for various graph classes. For instance, Matthews and Sumner~\cite{Matthews} proved that every $2$-connected claw-free graph with minimum degree at least $\frac{1}{3}(n-2)$ is Hamiltonian. Gerlach~\cite{Gerlach} established that every sep-chordal planar graph with toughness greater than one is Hamiltonian. For $2K_2$-free graphs, a series of improvements on the toughness bound has been obtained: Broersma, Patel and Pyatkin~\cite{Broersma and Patel} first showed that every $25$-tough $2K_2$-free graph on at least three vertices is Hamiltonian; Shan~\cite{Shan} then reduced the toughness requirement to $3$; and finally Ota and Sanka~\cite{Ota and Sanka} further lowered it to $2$. In addition, Shan~\cite{Shan p2p3} demonstrated that every $15$-tough $(P_2\cup P_3)$-free graph on at least three vertices is Hamiltonian. For $(P_2\cup kP_1)$-free graphs, Xu, Li and Zhou~\cite{Xu} proved that every $1$-tough $2k$-connected $(P_2\cup kP_1)$-free graph is Hamiltonian, while Ota and Sanka~\cite{Ota and SankaP1} established that every $1$-tough $k$-connected $(P_2\cup kP_1)$-free graph with minimum degree at least $\frac{3(k-1)}{2}$ is either Hamiltonian or the Petersen graph. Li, Broersma and Zhang~\cite{Li} showed that every $1$-tough $(P_3\cup P_1)$-free graph is Hamiltonian. Gao and Shan~\cite{Gao} further proved that every $7$-tough $(P_3\cup 2P_1)$-free graph is Hamiltonian. Despite these partial results, Conjecture 1.1 remains open. Motivated by these developments, in this paper we consider $(P_3\cup kP_1)$-free graphs under given conditions on connectivity and toughness, and we obtain the following results.
	
	\begin{restatable}{theorem}{thmOne}
		For an integer $k\ge 2$, if $G$ is a $2k$-connected, $1$-tough and  $(P_{3}\cup kP_{1})$-free graph, then $G$ is Hamiltonian.
	\end{restatable}
	
	Let $G$ be a $(P_3\cup 2P_1)$-free graph. If $G$ is $7$-tough, then $G$ is $14$-connected and $1$-tough. Consequently, $G$ satisfies the hypothesis of Theorem 1.1 with $k=2$ (i.e., $G$ is $4$-connected and $1$-tough), and therefore $G$ is Hamiltonian. Thus the main result in \cite{Gao} is a corollary of Theorem 1.1.
	
	\begin{corollary}(Gao and Shan~\cite{Gao})
		Every $7$-tough $(P_3\cup 2P_1)$-free graph on at least three vertices is Hamiltonian.
	\end{corollary}
	
	\begin{restatable}{theorem}{thmTwo}
		For an integer $k\ge 2$, if $G$ is a $(2k+1)$-connected $(P_{3}\cup kP_{1})$-free graph with toughness greater than $1$, then $G$ is Hamiltonian-connected.
	\end{restatable}

	In section 2, we present notations, terminology and a lemma about the properties of $(P_{3}\cup kP_{1})$-free graphs. In section 3, we provide the proofs of main results.
	
	\section{Preliminaries}
	The graphs considered in this paper are simple, finite and undirected. Let $G$ be a graph with vertex set $V(G)$ and edge set $E(G)$. For a vertex $v\in V(G)$, let $N_G(v)$ be its neighborhood, and for a set $S\subseteq V(G)$, let $N_G(S)=\bigcup_{v\in S}N_G(v)\setminus S$. Throughout this paper, we write $N(v)$ for $N_G(v)$ when no confusion arise. For a subgraph $H$ of $G$, set $N_H(v)=N_G(v)\cap V(H)$ and $N_H(S)=\bigcup_{v\in S}N_H(v)\setminus S$. Denote by $G[S]$ the subgraph induced by $S$, and let $G-S=G[V(G)\setminus S]$. The number of components of $G$ is denoted by $c(G)$. For graph-theoretical terminologies and notation not defined here, we refer to \cite{Bondy}.
	
	For a given graph $H$, a graph $G$ is called $H$-free if $G$ does not contain $H$ as an induced subgraph. For two vertex disjoint graphs $F_{1}$ and $F_{2}$, the disjoint union of graphs $F_{1}$ and $F_{2}$ is denoted by $F_{1}\cup F_{2}$. As usual, $P_{n}$ denotes the path on $n$ vertices. Also, we refer to $P_{n}$ as $n$-path. For positive integers $k$ and $l$, $kP_{l}$ denotes $k$ disjoint copies of the path $P_{l}$. For two distinct vertices $u$ and $v$ in $G$, a $(u,v)$-path $P$ is a path that has $u$ and $v$ as its endpoints. We can denote $P$ as $uPv$. In a graph $G$, a Hamiltonian path is a path that traverses every vertex exactly once, and a graph is called Hamiltonian-connected if it contains a Hamiltonian path between any two distinct vertices. A Hamiltonian cycle is a cycle that traverses every vertex exactly once, and a graph possessing such a cycle is a Hamiltonian graph. 
	
	Let $C$ be a cycle. We give $C$ a fixed orientation, for $u\in V(C)$ and a positive integer $j$, $u^{+j}$ and $u^{-j}$ denote the $j$-th successor and the $j$-th predecessor, respectively. Conveniently, we write $u^{+}$ and $u^{-}$ for $u^{+1}$ and $u^{-1}$, respectively. For $v,w\in V(C)$, $v\overrightarrow{C}w$ is a subpath of $C$ that denotes the $(v,w)$-path from $v$ to $w$ along the orientation of $C$. And $v\overleftarrow{C}w$ denotes $w\overrightarrow{C}v$. Let $V(v\overrightarrow{C}w)$ denote the set of all vertices of $v\overrightarrow{C}w$. The same notation conventions apply to a path $P$ that has been given  a fixed orientation.
	
	\begin{lemma}
		For an integer $k\ge 2$, let $G$ be a $(P_{3}\cup kP_{1})$-free graph. If $c(G-S)\ge k+1$ for some $S\subseteq V(G)$, then each component of $G-S$ is complete.
	\end{lemma}
	\noindent\textbf{Proof.} Suppose to the contrary that there exists a non complete component $H_{1}$ in $G-S$. Then $H_1$ contains a $3$-path $uvw$ with $uw\notin E(G)$. Since $c(G-S)\ge k+1$, we let $H_{2},H_{3},..,H_{k+1}$ be $k$ components of $G-S$ different from $H_{1}$. For any $i\in [2,k+1]$, let $z_{i}\in H_{i}$. Hence $\{u,v,w\}\cup \{z_{2},z_{3},..,z_{k+1}\}$ induces a $P_{3}\cup kP_{1}$, a contradiction. $\hfill\Box$

	\section{Main Results}
	
	\thmOne* 
	\noindent\textbf{Proof.} Suppose to the contrary that $G$ is not Hamiltonian. Let $C$ be the longest cycle of $G$ and $H$ be a component of $G-C$. Clearly, $|V(H)|\ge 1$. We give $C$ a fixed orientation. Let $H_{C}=N_{C}(H)=\{x_{1},x_{2},...,x_{h}\}$ and $H_{C}^{+}=\{x_{1}^{+},x_{2}^{+},...,x_{h}^{+}\}$. Without loss of generality, assume that $x_{1},x_{2},...,x_{h}$ appear in this order along the orientation of $C$. In the following proof, we denote the convention that indices are taken modulo $h$.

	\noindent{\bf Claim 1.}	For any $i\in [1,h]$, $x_{i}x_{i+1}\notin E(C)$ and $H_{C}^{+}$ is an independent set.
	\begin{proof}
		Suppose that $x_{i}x_{i+1}\in E(C)$ for some $i\in [1,h]$. Let $y_{i},y_{i+1}\in V(H)$ such that $x_{i}y_{i},x_{i+1}y_{i+1}\in E(G)$. Since $H$ is a component of $G-C$, there exists a $(y_{i},y_{i+1})$-path $P_{H}$ in $H$. Note that $y_{i}$ may be the same as $y_{i+1}$. Hence $x_{i}y_{i}P_{H}y_{i+1}x_{i+1}\overrightarrow {C} x_{i}$ is a longer cycle than $C$, a contradiction.
		
		Suppose that $x_{i}^{+}x_{j}^{+}\in E(G)$ for some distinct $i,j\in [1,h]$. Let $y_{i},y_{j}\in V(H)$ such that $x_{i}y_{i},x_{j}y_{j}\in E(G)$. Let $P_{H}'$ be a $(y_{i},y_{j})$-path in $H$. Without loss of generality, assume that $i<j$. Then $x_{i}y_{i}P_{H}'y_{j}x_{j}\overleftarrow {C} x_{i}^{+} x_{j}^{+} \overrightarrow {C} x_{i}$ is a longer cycle than $C$, a contradiction.
	\end{proof}
	
By Claim 1, we obtain that $V(C)\setminus H_{C}\ne \emptyset$. Thus, $c(G-H_{C})\ge 2$, i.e., $H_{C}$ is a cutset of $G$. Since $G$ is $2k$-connected, we have $|H_{C}|=h\ge \kappa(G) \ge 2k $.  
	
	\noindent{\bf Claim 2.}	Let $uv\in E(C)$, where $v=u^{+}$. If $ux_{i}^{+}\in E(G)$ for some $i\in [1,h]$, then $N(v)\cap (H_{C}^{+}\cap (V(x_{i}^{+}\overrightarrow {C}u)\setminus\{x_{i}^{+}\}))=\emptyset$.
	
	\begin{proof}
		Suppose to the contrary that $x_{j}^{+}\in N(v)\cap (H_{C}^{+}\cap (V(x_{i}^{+}\overrightarrow {C}u)\setminus\{x_{i}^{+}\}))$. Let $y_{i},y_{j}\in V(H)$ such that $x_{i}y_{i},x_{j}y_{j}\in E(G)$, and let $P_{H}$ be a $(y_{i},y_{j})$-path in $H$. Then $x_{i}y_{i}P_{H}y_{j}x_{j}\overleftarrow{C}x_{i}^{+}u\overleftarrow{C}x_{j}^{+}v\overrightarrow{C}x_{i}$ is a longer cycle than $C$, a contradiction.
	\end{proof}

	\noindent{\bf Claim 3.} $|N(x_{i})\cap H_{C}^{+}|\ge h-k+1$ for any $i\in [1,h]$.
	\begin{proof}
		Suppose to the contrary that $|N(x_{i})\cap H_{C}^{+}|\le h-k$ for some $i\in [1,h]$. Then $|H_{C}^{+}\setminus (N(x_{i})\cap H_{C}^{+})|\ge k$. We may let $x_{i_{1}}^{+},x_{i_{2}}^{+},...,x_{i_{k}}^{+}\in H_{C}^{+}\setminus (N(x_{i})\cap H_{C}^{+}) $ be $k$ distinct vertices. By Claim 1, we have $H_{C}\cap H_{C}^{+}=\emptyset$ and $H_{C}^{+}$ is an independent set. Let $y\in V(H)$ such that $x_{i}y\in E(G)$. Then $yx_{i}x_{i}^{+}$ is an induced $P_{3}$. Thus, $\{y,x_{i},x_{i}^{+}\}\cup \{x_{i_{1}}^{+},x_{i_{2}}^{+},...,x_{i_{k}}^{+}\}$ induces a $P_{3}\cup kP_{1}$, a contradiction.
	\end{proof}
	
	\noindent{\bf Claim 4.} For any $u\in V(C)$, $|N(u)\cap H_{C}^{+}|\le 1$ or $|N(u)\cap H_{C}^{+}|\ge h-k+1$.
	\begin{proof}
		Suppose to the contrary that $2\le |N(u)\cap H_{C}^{+}|\le h-k$ for some $u\in V(C)$. Then $|H_{C}^{+}\setminus (N(u)\cap H_{C}^{+})|\ge k$. Let  $x_{a}^{+},x_{b}^{+}\in N(u)\cap H_{C}^{+} $, and let $x_{i_{1}}^{+},x_{i_{2}}^{+},...,x_{i_{k}}^{+}\in H_{C}^{+}\setminus (N(u)\cap H_{C}^{+})$ be $k$ distinct vertices. Therefore, $\{x_{a}^{+},u,x_{b}^{+}\}\cup \{x_{i_{1}}^{+},x_{i_{2}}^{+},...,x_{i_{k}}^{+}\}$ induces a $P_{3}\cup kP_{1}$, a contradiction.
	\end{proof}
	
	For $u,v\in V(C)$, we know $u\overrightarrow{C}v$ is a subpath of $C$. If $G[V(u\overrightarrow{C}v)]$ is a clique and $|N(V(u\overrightarrow{C}v))\cap H_{C}^{+}|\le 1$, then we refer to such a subpath as a $1$-clique. For any $w\in V(C)$, if $|N(w)\cap H_{C}^{+}|\ge h-k+1$, then $w$ is denoted as $(h-k+1)$-vertex. By Claim 3, each vertex in $H_{C}$ is an $(h-k+1)$-vertex.

	\noindent{\bf Claim 5.} Let $u$ be an arbitrary vertex in $C$. Then the following statements hold.
	
	\noindent{\bf (\romannumeral1 )} If $|N(u)\cap H_{C}^{+}|\ge h-k+1$, then $|N(u^{-})\cap H_{C}^{+}|\le 1$ and $|N(u^{+})\cap H_{C}^{+}|\le 1$.
	
	\noindent{\bf (\romannumeral2)} If $|N(u)\cap H_{C}^{+}|\le 1$, then $G[V(u^{-j_{1}}\overrightarrow{C}u^{+j_{2}})\setminus \{u^{-j_{1}},u^{+j_{2}}\}]$ is a $1$-clique, where $j_{1}$ and $j_{2}$ are the smallest positive integers such that $|N(u^{-j_{1}})\cap H_{C}^{+}|\ge h-k+1$ and $|N(u^{+j_{2}})\cap H_{C}^{+}|\ge h-k+1$, respectively.
	
	\begin{proof}
		(\romannumeral1 ) Suppose to the contrary that $|N(u^{+})\cap H_{C}^{+}|\ge 2$ or $|N(u^{-})\cap H_{C}^{+}|\ge 2$. If $|N(u^{+})\cap H_{C}^{+}|\ge 2$, by Claim 4, we have $|N(u^{+})\cap H_{C}^{+}|\ge h-k+1$. Since $|N(u)\cap H_{C}^{+}|\ge h-k+1$, we have $|(N(u)\cap N(u^{+}))\cap H_{C}^{+}|\ge |N(u)\cap H_{C}^{+}|+|N(u^{+})\cap H_{C}^{+}|-|H_{C}^{+}|\ge h-2k+2\ge 2$. Without loss of generality, assume that $x_{a}^{+},x_{b}^{+}\in (N(u)\cap N(u^{+}))\cap H_{C}^{+} $ and $a>b$. Then $ux_{b}^{+}\in E(G)$ and $x_{a}^{+}\in N(u^{+})\cap (H_{C}^{+}\cap (V(x_{b}^{+}\overrightarrow {C}u)\setminus\{x_{b}^{+},u\}))$, which is a contradiction to Claim 2. Thus $|N(u^{+})\cap H_{C}^{+}|\le 1$. Similarly, we can prove that $|N(u^{-})\cap H_{C}^{+}|\le 1$.

		(\romannumeral2 ) If $|V(u^{-j_{1}}\overrightarrow{C}u^{+j_{2}})\setminus \{u^{-j_{1}},u^{+j_{2}}\}|=1$, then the result clearly holds. If $|V(u^{-j_{1}}\overrightarrow{C}u^{+j_{2}})\setminus \{u^{-j_{1}},u^{+j_{2}}\}|\ge 2$, then we suppose to the contrary that $G[V(u^{-j_{1}}\overrightarrow{C}u^{+j_{2}})\setminus \{u^{-j_{1}},u^{+j_{2}}\}]$ is not a $1$-clique. By the definition of $1$-clique, we obtain that $|N(V(u^{-j_{1}}\overrightarrow{C}u^{+j_{2}})\setminus \{u^{-j_{1}},u^{+j_{2}}\})\cap H_{C}^{+}|\ge 2$ or $G[V(u^{-j_{1}}\overrightarrow{C}u^{+j_{2}})\setminus \{u^{-j_{1}},u^{+j_{2}}\}]$ is not a clique. By the definition of $j_{1}$ and $j_{2}$, for any vertex $v\in V(u^{-j_{1}}\overrightarrow{C}u^{+j_{2}})\setminus \{u^{-j_{1}},u^{+j_{2}}\}$, we have $|N(v)\cap H_{C}^{+}|\le 1$.
		
		If $|N(V(u^{-j_{1}}\overrightarrow{C}u^{+j_{2}})\setminus \{u^{-j_{1}},u^{+j_{2}}\})\cap H_{C}^{+}|\ge 2$, let $x_{a}^{+}\in N(V(u^{-j_{1}}\overrightarrow{C}u^{+j_{2}})\setminus \{u^{-j_{1}},u^{+j_{2}}\})\cap H_{C}^{+} $. Since $|N(v)\cap H_{C}^{+}|\le 1$ for any vertex $v\in V(u^{-j_{1}}\overrightarrow{C}u^{+j_{2}})\setminus \{u^{-j_{1}},u^{+j_{2}}\}$, there exists an edge $wz\in E(G)$ in $G[V(u^{-j_{1}}\overrightarrow{C}u^{+j_{2}})\setminus \{u^{-j_{1}},u^{+j_{2}}\}]$ such that $x_{a}^{+}\in N(w)\cup N(z)$ and $x_{a}^{+}\notin N(w)\cap N(z)$. Without loss of generality, assume that $x_{a}^{+}\in N(w)$. Then $x_{a}^{+}wz$ is an induced $P_{3}$. Hence $|H_{C}^{+}\setminus ((N(w)\cap H_{C}^{+})\cup (N(z)\cap H_{C}^{+})) |\ge h-2\ge 2k-2\ge k$. Let $x_{i_{1}}^{+},x_{i_{2}}^{+},...,x_{i_{k}}^{+}\in H_{C}^{+}\setminus ((N(w)\cap H_{C}^{+})\cup (N(z)\cap H_{C}^{+}))$ be $k$ distinct vertices. Therefore $\{x_{a}^{+}, w ,z\} \cup \{x_{i_{1}}^{+},x_{i_{2}}^{+},...,x_{i_{k}}^{+}\}$ induces a $P_{3}\cup kP_{1}$, a contradiction.
		
		Thus, we assume that $|N(V(u^{-j_{1}}\overrightarrow{C}u^{+j_{2}})\setminus \{u^{-j_{1}},u^{+j_{2}}\})\cap H_{C}^{+}|\le 1$ and $G[V(u^{-j_{1}}\overrightarrow{C}u^{+j_{2}})\setminus \{u^{-j_{1}},u^{+j_{2}}\}]$ is not a clique, then it contains a $3$-path $P_{3}=v_{1}v_{2}v_{3}$ with $v_{1}v_{3}\notin E(G)$. Hence $|N(P_{3})\cap H_{C}^{+}|\le 1$. Let $x_{i_{1}}^{+},x_{i_{2}}^{+},...,x_{i_{k}}^{+}\in H_{C}^{+}\setminus (N(P_{3})\cap H_{C}^{+})$ be $k$ distinct vertices. Therefore, $\{v_{1},v_{2},v_{3}\}\cup \{x_{i_{1}}^{+},x_{i_{2}}^{+},...,x_{i_{k}}^{+}\}$ induces a $P_{3}\cup kP_{1}$, a contradiction.
	\end{proof}	
	Let $\mathcal{H}$ be the set of all maximal  $1$-cliques and $V_{h-k+1}$ be the set of all $(h-k+1)$-vertices on $C$. 
	By the proof of Claim 5, we can conclude that maximal $1$-clique and $(h-k+1)$-vertex are alternate on $C$. By this alternating property, we obtain that $|\mathcal{H}|=|V_{h-k+1}|$.
		
	 \noindent{\bf Claim 6.}  There is no edge between any two maximal $1$-cliques in $\mathcal{H}$.
	\begin{proof}
		Suppose to the contrary that there exists an edge $e=uv$ between some two maximal $1$-cliques in $\mathcal{H}$. Without loss of generality, assume that $uv$ is an edge between $H_{i_1}^{j_1}$ and $H_{i_2}^{j_2}$, $u\in V(H_{i_{1}}^{j_{1}})$, $v\in V(H_{i_{2}}^{j_{2}}) $ and $H_{i_{1}}^{j_{1}}$ precedes $H_{i_{2}}^{j_{2}}$ along the orientation of $C$. By the definition of $1$-clique, we have $|N(H_{i_{1}}^{j_{1}})\cap H_{C}^{+}|\le 1$ and $|N(H_{i_{2}}^{j_{2}})\cap H_{C}^{+}|\le 1$. 
		
		Then we can conclude that $G[V(H_{i_{1}}^{j_{1}}\cup H_{i_{2}}^{j_{2}})]$ is a clique. Otherwise, there exists a $3$-path $w_{1}w_{2}w_{3}$ in $G[V(H_{i_{1}}^{j_{1}}\cup H_{i_{2}}^{j_{2}})]$ with $w_{1}w_{3}\notin E(G)$. Since $|N(H_{i_{1}}^{j_{1}})\cap H_{C}^{+}|\le 1$ and $|N(H_{i_{2}}^{j_{2}})\cap H_{C}^{+}|\le 1$, we have $|N(\{w_{1},w_{2},w_{3}\})\cap H_{C}^{+}|\le 2$. Since $h\ge 2k$ and $k\ge 2$, we obtain $|H_{C}^{+}\setminus (N(\{w_{1},w_{2},w_{3}\})\cap H_{C}^{+})|\ge h-2\ge 2k-2\ge k$. Thus we may let $x_{i_{1}}^{+},x_{i_{2}}^{+},...,x_{i_{k}}^{+}\in H_{C}^{+}\setminus (N(\{w_{1},w_{2},w_{3}\})\cap H_{C}^{+})$ be $k$ distinct vertices. Then $\{w_{1},w_{2},w_{3}\}\cup \{x_{i_{1}}^{+},x_{i_{2}}^{+},...,x_{i_{k}}^{+}\}$ induces a $P_{3}\cup kP_{1}$, a contradiction. 
		
		Let $u_{1},u_{2}$ be two vertices of $H_{i_{1}}^{j_{1}}$ such that $|N(u_{1}^{-})\cap H_{C}^{+}|\ge h-k+1$ and $|N(u_{2}^{+})\cap H_{C}^{+}|\ge h-k+1$. Note that $u_{1}$ may be the same as $u_{2}$. Also, let $v_{1},v_{2}$ be two vertices of $H_{i_{2}}^{j_{2}}$ such that $|N(v_{1}^{-})\cap H_{C}^{+}|\ge h-k+1$ and $|N(v_{2}^{+})\cap H_{C}^{+}|\ge h-k+1$, where $v_{1}^{-}$ and $u_{2}^{+}$ may be the same. Therefore,  $|(N(u_{2}^{+})\cap N(v_{2}^{+}))\cap H_{C}^{+}|\ge |N(u_{2}^{+})\cap H_{C}^{+}|+|N(v_{2}^{+})\cap H_{C}^{+}|-|H_{C}^{+}|\ge h-2k+2\ge 2 $. Then let $x_{a}^{+},x_{b}^{+}$ be two distinct vertices in $(N(u_{2}^{+})\cap N(v_{2}^{+}))\cap H_{C}^{+}$. Without loss of generality, assume $a>b$. Let $y_{a},y_{b}\in V(H)$ with $x_{a}y_{a},x_{b}y_{b}\in E(G)$ and let $P_{H}$ be a $(y_{a},y_{b})$-path in $H$. Since $G[V(H_{i_{1}}^{j_{1}}\cup H_{i_{2}}^{j_{2}})]$ is a clique, there is a Hamiltonian path $P$ of $G[V(H_{i_{1}}^{j_{1}}\cup H_{i_{2}}^{j_{2}})]$ from $u_{1}$ to $v_{1}$.
		
		If $x_{a}^{+},x_{b}^{+}\in V(v_{2}^{+}\overrightarrow{C}u_{2})$, then $x_{a}y_{a}P_{H}y_{b}x_{b}\overleftarrow{C}v_{2}^{+}x_{a}^{+}\overrightarrow{C}u_{1}Pv_{1}\overleftarrow{C}u_{2}^{+}x_{b}^{+}\overrightarrow{C}x_{a} $ is a longer cycle than $C$, a contradiction.
		
		If $x_{a}^{+},x_{b}^{+}\in V(u_{2}^{+}\overrightarrow{C}v_{2})$, then $x_{a}y_{a}P_{H}y_{b}x_{b}\overleftarrow{C}u_{2}^{+}x_{a}^{+}\overrightarrow{C}v_{1}Pu_{1}\overleftarrow{C}v_{2}^{+}x_{b}^{+}\overrightarrow{C}x_{a} $ is a longer cycle than $C$, a contradiction.
		
		If $x_{a}^{+}\in V(v_{2}^{+}\overrightarrow{C}u_{2})$ and $x_{b}^{+}\in V(u_{2}^{+}\overrightarrow{C}v_{2})$, then $x_{a}y_{a}P_{H}y_{b}x_{b}\overleftarrow{C}u_{2}^{+}x_{b}^{+}\overrightarrow{C}v_{1}Pu_{1}\overleftarrow{C}x_{a}^{+}v_{2}^{+}\overrightarrow{C}x_{a}$ is a longer cycle than $C$, a contradiction.
	\end{proof}

	\noindent{\bf Claim 7.} Any two maximal $1$-cliques in $\mathcal{H}$ are contained in distinct components of $G-V_{h-k+1}$.
	\begin{proof}
		Suppose to the contrary that there exist two maximal $1$-cliques $H_{i_{1}}^{j_{1}}$ and $H_{i_{2}}^{j_{2}}$ that are contained in the same component of $G-V_{h-k+1}$. Without loss of generality, assume that $H_{i_{1}}^{j_{1}}$ precedes $H_{i_{2}}^{j_{2}}$ along the orientation of $C$. By Claim 6, we obtain that there exists a component $H'$ of $G-C$ such that $H'$ connects $H_{i_{1}}^{j_{1}}$ and $H_{i_{2}}^{j_{2}}$. Therefore, there exists at least one edge between $H'$ and $H_{i_{1}}^{j_{1}}$, and at least one edge between $H'$ and $H_{i_{2}}^{j_{2}}$. Let $u\in V(H_{i_{1}}^{j_{1}})$ and $v\in V(H_{i_{2}}^{j_{2}})$ such that $N(u)\cap V(H')\ne \emptyset $ and $N(v)\cap V(H')\ne \emptyset $. And let $P_{uv}$ be a shortest $(u,v)$-path whose all internal vertices belong to $H'$. Without loss of generality, assume that $P_{uv}=uz_{1}z_{2}... z_{s}v$. By the shortest path property of $P_{uv}$, we obtain that $uz_{1}z_{2}$ is an induced $3$-path. By the definition of $1$-clique, we have $|N(u)\cap H_{C}^{+}|\le 1$ and $|N(v)\cap H_{C}^{+}|\le 1$. And we can conclude that $|(N(z_{1})\cup N(z_{2}))\cap H_{C}^{+}|\ge h-k\ge k\ge 2$. Otherwise, $|(N(z_{1})\cup N(z_{2}))\cap H_{C}^{+}|\le h-k-1$. Then $|H_{C}^{+}\setminus (((N(z_{1})\cup N(z_{2}))\cap H_{C}^{+})\cup (N(u)\cap H_{C}^{+}))|\ge h-(h-k-1)-1=k $. Let $x_{i_{1}}^{+},x_{i_{2}}^{+},...,x_{i_{k}}^{+}\in H_{C}^{+}\setminus (((N(z_{1})\cup N(z_{2}))\cap H_{C}^{+})\cup (N(u)\cap H_{C}^{+}))$ be $k$ distinct vertices. Therefore, $\{u,z_{1},z_{2}\}\cup \{x_{i_{1}}^{+},x_{i_{2}}^{+},...,x_{i_{k}}^{+}\}$ induces a $P_{3}\cup kP_{1}$, a contradiction. Let $x_{a}^{+},x_{b}^{+}\in (N(z_{1})\cup N(z_{2}))\cap H_{C}^{+}$ be two distinct vertices. Without loss of generality, assume $a<b$. Let $y_{a},y_{b}\in V(H)$ with $x_{a}y_{a},x_{b}y_{b}\in E(G)$, and let $P_{ab}$ be a $(y_{a},y_{b})$-path in $H$. Since $|N(u)\cap H_{C}^{+}|\le 1 $, $|N(v)\cap H_{C}^{+}|\le 1$, $\{u,v\}\subseteq N_{C}(H')$ and $H_{C}\subseteq V_{h-k+1}$, we have $H\ne H'$. Hence, there exists an $(x_{a}^{+},x_{b}^{+})$-path $P_{H'}$ whose all internal vertex belong to $H'$ and are disjoint from $V(P_{ab})$. Therefore, $x_{a}y_{a}P_{ab}y_{b}x_{b}\overleftarrow{C}x_{a}^{+}P_{H'}x_{b}^{+}\overrightarrow{C}x_{a}$ is a longer cycle than $C$, a contradiction.
		\end{proof}
	 
	 Since $H_{C}^{+}$ is an independent set, no two vertices of $H_{C}^{+}$ belong to the same $1$-clique. By Claim 7, we obtain that $c(G-V_{h-k+1})\ge |\mathcal{H}|+1\ge |H_{C}^{+}|+1= h+1\ge 2k+1\ge 5$. Since $1$-clique and $(h-k+1)$-vertex are alternate on $C$, we have $|\mathcal{H}|=|V_{h-k+1}|$. Therefore $\tau (G)\le \frac{|V_{h-k+1}|}{c(G-V_{h-k+1})}\le \frac{|\mathcal{H}|}{|\mathcal{H}|+1}<1$, a contradiction.$\hfill\Box$

    \thmTwo*
	\noindent\textbf{Proof.} 
On the contrary, assume that there are two distinct vertices $u,v\in V(G)$ such that the longest $(u,v)$-path $P$ in $G$ is not a Hamiltonian path. Then $V(G)\setminus V(P)\ne \emptyset$. Let $H$ be a component of $G-P$.  Clearly, $|V(H)|\ge 1$. We give $P$ an orientation, which is starting from $u$ and ending at $v$. Let $H_{P}=N_{P}(H)=\{x_{1},x_{2},...,x_{h}\}$. Without loss of generality, assume that $x_{1},x_{2},...,x_{h}$ appear in this order along the orientation of $P$. Let $H_{P}^{+}=\{x_{1}^{+},x_{2}^{+},...,x_{h-1+\ell}^{+}\}$, where $\ell =1$ if $x_{h}\ne v$ and $\ell =0$ otherwise. 
	
	\noindent{\bf Claim 1.}	For any $i\in [1,h-1]$, $x_{i}x_{i+1}\notin E(P)$ and $H_{P}^{+}$ is an independent set.
	\begin{proof}
		Suppose that $x_{i}x_{i+1}\in E(P)$ for some $i\in [1,h-1]$. Let $y_{i},y_{i+1}\in V(H)$ such that $x_{i}y_{i},x_{i+1}y_{i+1}\in E(G)$. Since $H$ is a component of $G-P$, there exists a $(y_{i},y_{i+1})$-path $P_{H}$ in $H$. Note that $y_{i}$ may be the same as $y_{i+1}$. Hence $u\overrightarrow {P}x_{i}y_{i}P_{H}y_{i+1}x_{i+1}\overrightarrow {P}v$ is a longer $(u,v)$-path than $P$, a contradiction.
		
		Suppose that $x_{i}^{+}x_{j}^{+}\in E(G)$ for some distinct $i,j\in [1,h-1+\ell]$. Let $y_{i},y_{j}\in V(H)$ such that $x_{i}y_{i},x_{j}y_{j}\in E(G)$, and let $P_{H}'$ be a $(y_{i},y_{j})$-path in $H$. Without loss of generality, assume that $i<j$. Then $u\overrightarrow{P}x_{i}y_{i}P_{H}'y_{j}x_{j}\overleftarrow {P} x_{i}^{+} x_{j}^{+} \overrightarrow {P} v$ is a longer $(u,v)$-path than $P$, a contradiction.
	\end{proof}
By Claim 1, we obtain that $V(P)\setminus H_{P}\ne \emptyset$. Thus, $c(G-H_{P})\ge 2$, i.e., $H_{P}$ is a cutset of $G$. Since $\kappa(G)\ge 2k+1$, we have $|H_{P}|=h\ge \kappa(G) \ge 2k+1$. 
	
	\noindent{\bf Claim 2.}	Let $wz\in E(P)$, where $z=w^{+}$. If there exists a vertex $x_{a}^{+}\in H_{P}^{+}$ such that $wx_{a}^{+}\in E(G)$, then the following statements hold.
	
	\noindent{\bf (\romannumeral1)} If $x_{a}^{+}\in V(u\overrightarrow{P}w)$, then $N(z)\cap (H_P^+\cap (V(x_a^+\overrightarrow{P}w)\setminus \{x_a^+\}))=\emptyset$.

	\noindent{\bf (\romannumeral2)} If $x_{a}^{+}\in V(w\overrightarrow{P}v)$, then $N(z)\cap (H_{P}^{+}\cap (V(u\overrightarrow{P}w)\cup (V(x_{a}^{+}\overrightarrow{P}v)\setminus \{x_{a}^{+}\})))=\emptyset $.

	\begin{proof}
		\noindent{ (\romannumeral1)} Suppose to the contrary that there exists a vertex $x_{b}^{+}\in N(z)\cap (H_{P}^{+}\cap (V(x_{a}^{+}\overrightarrow{P}w)\setminus \{x_{a}^{+}\}))$. Let $y_{a},y_{b}\in V(H)$ such that $x_{a}y_{a},x_{b}y_{b}\in E(G)$, and let $P_{H}$ be a $(y_{a},y_{b})$-path in $H$. Then $u\overrightarrow{P}x_{a}y_{a}P_{H}y_{b}x_{b}\overleftarrow{P}x_{a}^{+}w\overleftarrow{P}x_{b}^{+}z\overrightarrow{P}v$ is a longer $(u,v)$-path than $P$, a contradiction.
		
		\noindent{ (\romannumeral2)} Suppose to the contrary that there exists a vertex $x_{b}^{+}\in  N(z)\cap (H_{P}^{+}\cap (V(u\overrightarrow{P}w)\cup (V(x_{a}^{+}\overrightarrow{P}v)\setminus \{x_{a}^{+}\}))) $. Let $y_{a},y_{b}\in V(H)$ such that $x_{a}y_{a},x_{b}y_{b}\in E(G)$, and let $P_{H}$ be a $(y_{a},y_{b})$-path in $H$. If $x_{b}^{+}\in V(u\overrightarrow{P}w)$, then $u\overrightarrow{P}x_{b}y_{b}P_{H}y_{a}x_{a}\overleftarrow{P}zx_{b}^{+}\overrightarrow{P}wx_{a}^{+}\overrightarrow{P}v$ is a longer $(u,v)$-path than $P$, a contradiction. If $x_{b}^{+}\in V(x_{a}^{+}\overrightarrow{P}v)\setminus \{x_{a}^{+}\}$, then $u\overrightarrow{P}wx_{a}^{+}\overrightarrow{P}x_{b}y_{b}P_{H}y_{a}x_{a}\overleftarrow{P}zx_{b}^{+}\overrightarrow{P}v$ is a longer $(u,v)$-path than $P$, a contradiction.
	\end{proof}
	
	\noindent{\bf Claim 3.} For any integer $i\in [1,h-1+\ell]$, $|N(x_{i})\cap H_{P}^{+}|\ge h-k+\ell$.
	\begin{proof}
		Suppose to the contrary that $|N(x_{i})\cap H_{P}^{+}|\le h-k+\ell-1$ for some $i\in [1,h-k+\ell]$. Since $|H_{P}^{+}\setminus (N(x_{i})\cap H_{P}^{+})|\ge h-1+\ell -(h-k+\ell-1)=k$, there exist $k$ distinct vertices $x_{i_{1}}^{+},x_{i_{2}}^{+},...,x_{i_{k}}^{+}$ in $H_{P}^{+}\setminus (N(x_{i})\cap H_{P}^{+}) $. By Claim 1, $H_{P}\cap H_{P}^{+}=\emptyset$ and $H_{P}^{+}$ is an independent set. Let $y\in V(H)$ such that $x_{i}y\in E(G)$. Then $yx_{i}x_{i}^{+}$ is an induced $P_{3}$. Thus, $\{y,x_{i},x_{i}^{+}\}\cup \{x_{i_{1}}^{+},x_{i_{2}}^{+},...,x_{i_{k}}^{+}\}$ induces a $P_{3}\cup kP_{1}$, a contradiction.
	\end{proof}
	
	\noindent{\bf Claim 4.} For any vertex $w\in V(P)$, $|N(w)\cap H_{P}^{+}|\le 1$ or $|N(w)\cap H_{P}^{+}|\ge h-k+\ell$.
	\begin{proof}
	
	Suppose to the contrary that $2\le |N(w)\cap H_{P}^{+}|\le h-k+\ell-1$ for some $w\in V(P)$. Then $|H_{P}^{+}\setminus (N(w)\cap H_{P}^{+})|\ge h-1+\ell -(h-k+\ell-1)=k$. Let $x_{a}^{+},x_{b}^{+}\in N(w)\cap H_{P}^{+}$, and let $x_{i_{1}}^{+},x_{i_{2}}^{+},...,x_{i_{k}}^{+}\in H_{P}^{+}\setminus (N(w)\cap H_{P}^{+}) $ be $k$ distinct vertices. Therefore, $\{x_{a}^{+},w,x_{b}^{+}\}\cup \{x_{i_{1}}^{+},x_{i_{2}}^{+},...,x_{i_{k}}^{+}\}$ induces a $P_{3}\cup kP_{1}$, a contradiction.
	\end{proof}

 Let $w_{1}\overrightarrow{P}w_{2}$ be a subpath of $P$. We call  $G[V(w_{1}\overrightarrow{P}w_{2})]$   a $1$-clique if $G[V(w_{1}\overrightarrow{P}w_{2})]$ is a clique and $|N(V(w_{1}\overrightarrow{P}w_{2}))\cap H_{P}^{+}|\le 1$. For any $w\in V(P)$, if $|N(w)\cap H_{P}^{+}|\ge h-k+\ell$, then $w$ is denoted as a $(h-k+\ell)$-vertex. By Claim 3, each vertex in $H_{P}$ is $(h-k+\ell)$-vertex.

 \noindent{\bf Claim 5.} Let $w$ be an arbitrary vertex in $P$ with $|N(w)\cap H_{P}^{+}|=0$, and let $w_{1}\overrightarrow{P}w_{2}$ be a subpath of $P$ containing $w$. If each vertex $z\in V(w_{1}\overrightarrow{P}w_{2})$ satisfies $|N(z)\cap H_{P}^{+}|\le 1$, then $|N(V(w_{1}\overrightarrow{P}w_{2}))\cap H_{P}^{+}|=0$ and $G[V(w_{1}\overrightarrow{P}w_{2})]$ is a $1$-clique.
\begin{proof}
	 If $|V(w_{1}\overrightarrow{P}w_{2})|=1$, the conclusion is immediate. If $|V(w_{1}\overrightarrow{P}w_{2})|\ge 2$ and $w^{+}$ exists, we suppose that $|N(w^{+})\cap H_{P}^{+}|= 1$. Without loss of generality, assume that $N(w^{+})\cap H_{P}^{+}= \{x_{a}^{+}\}$. Then $wx_{a}^{+}\notin E(G)$ and $ww^{+}x_{a}^{+}$ is an induced $3$-path. Since $h\ge 2k+1$ and $k\ge 2$, we have $|H_{P}^{+}\setminus ((N(w^{+})\cap H_{P}^{+}) )|= h-1+\ell-1\ge 2k-1\ge k$. Let $x_{i_{1}}^{+},x_{i_{2}}^{+},...,x_{i_{k}}^{+}\in H_{P}^{+}\setminus ((N(w^{+})\cap H_{P}^{+}) )$ be $k$ distinct vertices. Therefore $\{w,w^{+},x_{a}^{+}\}\cup \{x_{i_{1}}^{+},x_{i_{2}}^{+},...,x_{i_{k}}^{+}\}$ induces a $P_{3}\cup kP_{1}$, a contradiction. Hence $|N(w^{+})\cap H_{P}^{+}|=0$. By repeating this argument, it follows that $|N(z)\cap H_P^+|=0$ for every $z\in V(w\overrightarrow{P}w_2)$. Then, $|N(V(w\overrightarrow{P}w_{2}))\cap H_{P}^{+}|=0$. By a symmetric argument, we can obtain that $|N(V(w_{1}\overrightarrow{P}w))\cap H_{P}^{+}|=0$ if $|V(w_{1}\overrightarrow{P}w_{2})|\ge 2$ and $w^{-}$ exists. Thus $|N(V(w_{1}\overrightarrow{P}w_{2}))\cap H_{P}^{+}|=0$.
	
	It remains to show that $G[V(w_{1}\overrightarrow{P}w_{2})]$ is a clique. Suppose not, then $G[V(w_{1}\overrightarrow{P}w_{2})]$ contains a $3$-path $P_{3}=z_{1}z_{2}z_{3}$ with $z_{1}z_{3}\notin E(G)$. Since $|N(V(w_{1}\overrightarrow{P}w_{2}))\cap H_{P}^{+}|=0$, we have $|N(P_{3})\cap H_{P}^{+}|= 0$. Let $x_{i_{1}}^{+},x_{i_{2}}^{+},...,x_{i_{k}}^{+}\in H_{P}^{+}$ be $k$ distinct vertices. Thus, $\{z_{1},z_{2},z_{3}\}\cup \{x_{i_{1}}^{+},x_{i_{2}}^{+},...,x_{i_{k}}^{+}\}$ induces a $P_{3}\cup kP_{1}$, a contradiction. Then $G[V(w_{1}\overrightarrow{P}w_{2})]$ is a clique.
\end{proof}

\noindent{\bf Claim 6.} Let $w$ be an arbitrary vertex in $P$ with $|N(w)\cap H_P^+|=1$, and write $N(w)\cap H_P^+ =\{x_a^+\}$. Let $w_1\overrightarrow{P}w_2$ be a subpath of $P$ containing $w$. If each vertex $z\in V(w_1\overrightarrow{P}w_2)$ satisfies $|N(z)\cap H_P^+|\le 1$, then $N(V(w_{1}\overrightarrow{P}w_{2}))\cap H_{P}^{+}=\{x_a^+\}$ and $G[V(w_1\overrightarrow{P}w_2)]$ is a $1$-clique.
	\begin{proof}
		If $|V(w_1\overrightarrow{P}w_2)|=1$, the conclusion is immediate. If $|V(w_1\overrightarrow{P}w_2)|\ge 2$ and $w^{+}$ exists, we suppose that $N(w^{+})\cap H_P^+ \ne \{x_a^+\}$. Then there exists an induced $3$-path $x_{a}^{+}ww^{+}$ with $x_{a}^{+}w^{+}\notin E(G)$. Since $|N(z)\cap H_P^+|\le 1$ for each vertex $z\in V(w_1\overrightarrow{P}w_2)$, we obtain that $|N(w^{+})\cap H_P^+|\le 1$. Due to $h\ge 2k+1$ and $k\ge 2$, we have $|H_{P}^{+}\setminus ((N(w^{+})\cap H_{P}^{+})\cup\{x_{a}^{+}\} )|\ge h-1+\ell-2\ge 2k-2\ge k$. Let $x_{i_{1}}^{+},x_{i_{2}}^{+},...,x_{i_{k}}^{+}\in H_{P}^{+}\setminus ((N(w^{+})\cap H_{P}^{+})\cup\{x_{a}^{+}\} )$ be $k$ distinct vertices. Then $\{x_{a}^{+},w,w^{+}\}\cup \{x_{i_{1}}^{+},x_{i_{2}}^{+},...,x_{i_{k}}^{+}\}$ induces a $P_{3}\cup kP_{1}$, a contradiction. Thus, $N(w^{+})\cap H_{P}^{+}=\{x_{a}^{+}\}$. By repeating this argument, it follows that $N(z)\cap H_{P}^{+}=\{x_{a}^{+}\}$ for every vertex $z\in V(w\overrightarrow{P}w_{2})$. Then $N(V(w\overrightarrow{P}w_{2}))\cap H_{P}^{+}=\{x_{a}^{+}\}$. By a symmetric argument, we can obtain that $N(V(w_{1}\overrightarrow{P}w))\cap H_{P}^{+}=\{x_{a}^{+}\}$ if $|V(w_{1}\overrightarrow{P}w_{2})|\ge 2$ and $w^{-}$ exists. Thus $N(V(w_{1}\overrightarrow{P}w_{2}))\cap H_{P}^{+}=\{x_{a}^{+}\}$.

		 It remains to show that $G[V(w_{1}\overrightarrow{P}w_{2})]$ is a clique. By a similar analysis of Claim 5, we can obtain that $G[V(w_{1}\overrightarrow{P}w_{2})]$ is a clique.
	\end{proof}
	For any vertex $w\in V(P)$ with $|N(w)\cap H_{P}^{+}|\le 1$, let $w_{L}$ be the vertex $w^{-j_1}$ if there exists a smallest positive integer $j_1$ such that $|N(w^{-j_1})\cap H_P^+|\ge h-k+\ell$, otherwise let $w_{L}=u$. Let $w_{R}$ be the vertex $w^{+j_2}$ if there exists a smallest positive integer $j_2$ such that $|N(w^{+j_2})\cap H_P^+|\ge h-k+\ell$, otherwise let $w_{R}=v$. Set $S = \{w_{L} : w_{L}\neq u\} \cup \{w_{R} : w_{R}\neq v\}$. By the definition of $w_{L}$ and $w_{R}$, we obtain that each vertex $w'\in V(w_{L}\overrightarrow{P}w_{R})\setminus S$ satisfies $|N(w')\cap H_{P}^{+}|\le 1$. Therefore, by Claim 5 and 6, we obtain that $|N(V(w_{L}\overrightarrow{P}w_{R})\setminus S)\cap H_P^+|\le 1$ and $G[V(w_{L}\overrightarrow{P}w_{R})\setminus S]$ is a $1$-clique.

	\noindent{\bf Claim 7.} For any vertex $w\in V(P)$, if $|N(w)\cap H_{P}^{+}|\ge h-k+\ell$ and both $w^{-}$ and $w^{+}$ exist, then $|N(w^{-})\cap H_{P}^{+}|\le 1$ and $|N(w^{+})\cap H_{P}^{+}|\le 1$.
		\begin{proof}
		Suppose that $|N(w^{+})\cap H_{P}^{+}|\ge 2$. By Claim 4, we have $|N(w^{+})\cap H_{P}^{+}|\ge h-k+\ell$. Then $|(N(w)\cap N(w^{+}))\cap H_{P}^{+}|\ge h-2k+\ell+1\ge 2$. Let $x_{a}^{+},x_{b}^{+}\in (N(w)\cap N(w^{+}))\cap H_{P}^{+} $. Without loss of generality, assume that $a<b$. If $x_{a}^{+},x_{b}^{+}\in V(u\overrightarrow{P}w)$ or $x_{a}^{+},x_{b}^{+}\in V(w\overrightarrow{P}v)$, since $wx_{a}^{+},w^{+}x_{b}^{+}\in E(G)$, we have $x_{b}^{+}\in N(w^{+})\cap (H_P^+\cap (V(x_a^+\overrightarrow{P}w)\setminus \{x_a^+\}))$ or $x_{b}^{+}\in N(w^{+})\cap (H_{P}^{+}\cap ( V(x_{a}^{+}\overrightarrow{P}v)\setminus \{x_{a}^{+}\}))$, which is a contradiction to Claim 2. If $x_{a}^{+}\in V(u\overrightarrow{P}w)$ and $x_{b}^{+}\in V(w\overrightarrow{P}v)$, since $wx_{b}^{+},w^{+}x_{a}^{+}\in E(G)$, we have $x_{a}^{+}\in N(w^{+})\cap (H_{P}^{+}\cap  V(u\overrightarrow{P}w))$, which is also a contradiction to Claim 2. By a symmetric argument, we can obtain that $|N(w^{-})\cap H_{P}^{+}|\le 1$.
	\end{proof}
	
		Let $\mathcal{H}$ be the set of all maximal  $1$-cliques and $V_{h-k+\ell}$ be the set of all $(h-k+\ell)$-vertices on $P$. 
	By the proofs of Claim 5-7, we can conclude that maximal $1$-clique and $(h-k+\ell)$-vertex are alternate on $P$.

	\noindent{\bf Claim 8.} Let $H_{i_1}^{j_1}$ and $H_{i_2}^{j_2}$ be any two distinct maximal 1-cliques in $\mathcal{H}$. If at most one of $H_{i_{1}}^{j_{1}}$ and $H_{i_{2}}^{j_{2}}$ contains $u$ or $v$, then there is no edge between $H_{i_{1}}^{j_{1}}$ and $H_{i_{2}}^{j_{2}}$.

	\begin{proof}
		Suppose to the contrary that there exists an edge $e=wz$ between $H_{i_{1}}^{j_{1}}$ and $H_{i_{2}}^{j_{2}}$. Without loss of generality, assume that $w\in V(H_{i_{1}}^{j_{1}})$, $z\in V(H_{i_{2}}^{j_{2}}) $ and $H_{i_{1}}^{j_{1}}$ precedes $H_{i_{2}}^{j_{2}}$ along the orientation of $P$. By the definition of $1$-clique, we have $|N(H_{i_{1}}^{j_{1}})\cap H_{P}^{+}|\le 1$ and $|N(H_{i_{2}}^{j_{2}})\cap H_{P}^{+}|\le 1$. 
		
		Then we can conclude that $G[V(H_{i_{1}}^{j_{1}}\cup H_{i_{2}}^{j_{2}})]$ is a clique. Otherwise, there exists a $3$-path $w_{1}w_{2}w_{3}$ in $G[V(H_{i_{1}}^{j_{1}}\cup H_{i_{2}}^{j_{2}})]$. Since $|N(H_{i_{1}}^{j_{1}})\cap H_{P}^{+}|\le 1$ and $|N(H_{i_{2}}^{j_{2}})\cap H_{P}^{+}|\le 1$, we have $|N(\{w_{1},w_{2},w_{3}\})\cap H_{P}^{+}|\le 2$. Since $h\ge 2k+1$ and $k\ge 2$, we may let $x_{i_{1}}^{+},x_{i_{2}}^{+},...,x_{i_{k}}^{+}\in H_{P}^{+}\setminus (N(\{w_{1},w_{2},w_{3}\})\cap H_{P}^{+})$ be $k$ distinct vertices. Hence, $\{w_{1},w_{2},w_{3}\}\cup \{x_{i_{1}}^{+},x_{i_{2}}^{+},...,x_{i_{k}}^{+}\}$ induces a $P_{3}\cup kP_{1}$, a contradiction.
		
		Let $u_{1},u_{2}$ be two vertices of $H_{i_{1}}^{j_{1}}$ such that $|N(u_{1}^{-})\cap H_{P}^{+}|\ge h-k+\ell$ and $|N(u_{2}^{+})\cap H_{P}^{+}|\ge h-k+\ell$. Note that $u_{1}^{-}$ does not exist if $u\in V(H_{i_{1}}^{j_{1}})$. Also, let $v_{1},v_{2}$ be two vertices of $H_{i_{2}}^{j_{2}}$ such that $|N(v_{1}^{-})\cap H_{P}^{+}|\ge h-k+\ell$ and $|N(v_{2}^{+})\cap H_{P}^{+}|\ge h-k+\ell$. Note that $v_{2}^{+}$ does not exist if $v\in V(H_{i_{2}}^{j_{2}})$.
		
		If $v_{2}^{+}$ exists, then $|(N(v_{2}^{+})\cap N(u_{2}^{+}))\cap H_{P}^{+}|\ge h-2k+\ell+1\ge 2$. Let $x_{a}^{+},x_{b}^{+}\in (N(v_{2}^{+})\cap N(u_{2}^{+}))\cap H_{P}^{+} $. Without loss of generality, assume that $a<b$. Let $y_{a},y_{b}\in V(H)$ with $x_{a}y_{a},x_{b}y_{b}\in E(G)$, and let $P_{ab}$ be a $(y_{a},y_{b})$-path in $H$. Since $G[V(H_{i_{1}}^{j_{1}}\cup H_{i_{2}}^{j_{2}})]$ is a clique, there exists a Hamiltonian path $P_{u_{1}v_{1}}$ of $G[V(H_{i_{1}}^{j_{1}}\cup H_{i_{2}}^{j_{2}})]$ from $u_{1}$ to $v_{1}$. Note that we let $u_{1}=u$ if $u\in V(H_{i_{1}}^{j_{1}})$.
		
		 When $x_{a}^{+},x_{b}^{+}\in V(u\overrightarrow{P}u_{2} )$, then $u\overrightarrow{P}x_{a}y_{a}P_{ab}y_{b}x_{b}\overleftarrow{P}x_{a}^{+}u_{2}^{+}\overrightarrow{P}v_{1}P_{u_{1}v_{1}}u_{1}\overleftarrow{P}x_{b}^{+}v_{2}^{+}\overrightarrow{P}v$ is a longer $(u,v$)-path than $P$, a contradiction.
		 
		  When $x_{a}^{+},x_{b}^{+}\in V(u_{2}^{+}\overrightarrow{P}v_{1}^{-})$, then $u\overrightarrow{P}u_{1}P_{u_{1}v_{1}}v_{1}\overleftarrow{P}x_{b}^{+}u_{2}^{+}\overrightarrow{P}x_{a}y_{a}P_{ab}y_{b}x_{b}\overleftarrow{P}x_{a}^{+}v_{2}^{+}\overrightarrow{P}v$ is a longer $(u,v)$-path than $P$, a contradiction. 
		  
		  When $x_{a}^{+},x_{b}^{+}\in V(v_{1} \overrightarrow{P}v)$, then $u\overrightarrow{P}u_{1}P_{u_{1}v_{1}}v_{1}\overleftarrow{P}u_{2}^{+}x_{a}^{+}\overrightarrow{P}x_{b}\allowbreak y_{b}P_{ab}y_{a}x_{a}\overleftarrow{P}v_{2}^{+}x_{b}^{+}\overrightarrow{P}v$ is a longer $(u,v$)-path than $P$, a contradiction. 
		  
		  When $x_{a}^{+}\in V(u\overrightarrow{P}u_{2} )$ and $x_{b}^{+}\in V(u_{2}^{+}\overrightarrow{P}v_{1}^{-})$, then $u\overrightarrow{P}x_{a}y_{a}P_{ab}y_{b}x_{b}\overleftarrow{P}u_{2}^{+}x_{a}^{+}\overrightarrow{P}u_{1}P_{u_{1}v_{1}}v_{1}\overleftarrow{P}x_{b}^{+}\allowbreak v_{2}^{+}\overrightarrow{P}v$ is a longer $(u,v$)-path than $P$, a contradiction. 
		  
		  When $x_{a}^{+}\in V(u\overrightarrow{P}u_{2} )$ and $x_{b}^{+}\in V(v_{1} \overrightarrow{P}v)$, then $u\overrightarrow{P}x_{a}y_{a}P_{ab}y_{b}x_{b}\overleftarrow{P}v_{2}^{+}\allowbreak x_{a}^{+}\overrightarrow{P}u_{1}P_{u_{1}v_{1}}v_{1}\overleftarrow{P}u_{2}^{+}x_{b}^{+}\allowbreak\overrightarrow{P}v$ is a longer $(u,v$)-path than $P$, a contradiction. 
		  
		  When $x_{a}^{+}\in V(u_{2}^{+}\overrightarrow{P}v_{1}^{-})$ and $x_{b}^{+}\in V(v_{1} \overrightarrow{P}v)$, then $u\overrightarrow{P}u_{1}P_{u_{1}v_{1}}v_{1}\overleftarrow{P}x_{a}^{+}u_{2}^{+}\overrightarrow{P}x_{a}y_{a}P_{ab}y_{b}x_{b}\overleftarrow{P}v_{2}^{+}\allowbreak x_{b}^{+}\overrightarrow{P}v$ is a longer $(u,v$)-path than $P$, a contradiction.
		
		If $v_{2}^{+}$ does not exist, then $v\in V(H_{i_{2}}^{j_{2}})$ and $u\notin V(H_{i_{1}}^{j_{1}})$. Since $G[V(H_{i_{1}}^{j_{1}}\cup H_{i_{2}}^{j_{2}})]$ is a clique, there exists a Hamiltonian path $P_{u_{2}v}$ of $G[V(H_{i_{1}}^{j_{1}}\cup H_{i_{2}}^{j_{2}})]$ from $u_{2}$ to $v$ or there exists a Hamiltonian path $P_{u_{1}v}$ of $G[V(H_{i_{1}}^{j_{1}}\cup H_{i_{2}}^{j_{2}})]$ from $u_{1}$ to $v$. Since $|N(u_{1}^{-})\cap H_{P}^{+}|\ge h-k+\ell$ and $|N(v_{1}^{-})\cap H_{P}^{+}|\ge h-k+\ell$, it follows that $|(N(u_{1}^{-})\cap N(v_{1}^{-}))\cap H_{P}^{+}|\ge h-2k+\ell+1\ge 2$. Let $x_{c}^{+},x_{d}^{+}\in (N(u_{1}^{-})\cap N(v_{1}^{-}))\cap H_{P}^{+} $. Without loss of generality, assume that $c<d$. Let $y_{c},y_{d}\in V(H)$ with $x_{c}y_{c},x_{d}y_{d}\in E(G)$, and let $P_{cd}$ be a $(y_{c},y_{d})$-path in $H$.
		
		 When $x_{c}^{+},x_{d}^{+}\in V(u\overrightarrow{P}u_{2} )$, then $u\overrightarrow{P}x_{c}y_{c}P_{cd}y_{d}x_{d}\overleftarrow{P}x_{c}^{+}u_{1}^{-}\overleftarrow{P}x_{d}^{+}v_{1}^{-}\overleftarrow{P}u_{2}P_{u_{2}v}v$ is a longer $(u,v$)-path than $P$, a contradiction.
		 
		  When $x_{c}^{+},x_{d}^{+}\in V(u_{2}^{+}\overrightarrow{P}v)$, then $u\overrightarrow{P}u_{1}^{-}x_{d}^{+}\overrightarrow{P}v_{1}^{-}x_{c}^{+}\overrightarrow{P}x_{d}y_{d}P_{cd}y_{c}x_{c}\overleftarrow{P}u_{2}P_{u_{2}v}v$ is a longer $(u,v$)-path than $P$, a contradiction. 
		  
		  When $x_{c}^{+}\in V(u\overrightarrow{P}u_{2} )$, $x_{d}^{+}\in V(u_{2}^{+}\overrightarrow{P}v)$ and $(N(u_{2}^{+})\cap H_{P}^{+})\cap V(u_{2}^{+}\overrightarrow{P}v  )\ne \emptyset$, we let $x_{e}^{+}\in (N(u_{2}^{+})\cap H_{P}^{+})\cap V(u_{2}^{+}\overrightarrow{P}v )$. Let $y_{e}\in V(H)$ with $x_{e}y_{e}\in E(G)$, and let $P_{ce}$ be a $(y_{c},y_{e})$-path in $H$. Thus, $u\overrightarrow{P}x_{c}y_{c}P_{ce}y_{e}x_{e}\overleftarrow{P}u_{2}^{+}x_{e}^{+}\overrightarrow{P}v_{1}^{-}x_{c}^{+}\overrightarrow{P}u_{1}P_{u_{1}v}v$ is a longer $(u,v$)-path than $P$, a contradiction. 
		  
		  When $x_{c}^{+}\in V(u\overrightarrow{P}u_{2} )$, $x_{d}^{+}\in V(u_{2}^{+}\overrightarrow{P}v )$ and $(N(u_{2}^{+})\cap H_{P}^{+})\cap V(u_{2}^{+}\overrightarrow{P}v )= \emptyset$, then $|(N(u_{2}^{+})\cap H_{P}^{+})\cap V(u\overrightarrow{P}u_{2})|\ge h-k+\ell\ge 3 $. Let $x_{e}^{+}\in (N(u_{2}^{+})\cap H_{P}^{+})\cap V(u\overrightarrow{P}u_{2})$ such that $e\ne c$. Let $y_{e}\in V(H)$ with $x_{e}y_{e}\in E(G)$, and let $P_{ce}$ be a $(y_{c},y_{e})$-path in $H$. Thus, $u\overrightarrow{P}x_{c}y_{c}P_{ce}y_{e}x_{e}\overleftarrow{P}x_{c}^{+}v_{1}^{-}\overleftarrow{P}u_{2}^{+}\allowbreak x_{e}^{+} \overrightarrow{P}u_{1}P_{u_{1}v}v$ or $u\overrightarrow{P}x_{e}y_{e}P_{ce}y_{c}x_{c}\overleftarrow{P}x_{e}^{+}u_{2}^{+}\overrightarrow{P}v_{1}^{-}x_{c}^{+}\overrightarrow{P}u_{1}P_{u_{1}v}v$ is a longer $(u,v$)-path than $P$, a contradiction.
	\end{proof}

\noindent{\bf Claim 9.} Let $H_{i_1}^{j_1}$ and $H_{i_2}^{j_2}$ be any two distinct maximal 1-cliques in $\mathcal{H}$. If at most one of $H_{i_{1}}^{j_{1}}$ and $H_{i_{2}}^{j_{2}}$ contains $u$ or $v$, then $H_{i_{1}}^{j_{1}}$ and $H_{i_{2}}^{j_{2}}$ are contained in distinct components of $G-V_{h-k+\ell}$.
\begin{proof}
	Suppose to the contrary that $H_{i_{1}}^{j_{1}}$ and $H_{i_{2}}^{j_{2}}$ are contained in the same component of $G-V_{h-k+\ell}$. Without loss of generality, assume that $H_{i_1}^{j_{1}}$ precedes $H_{i_2}^{j_{2}}$ along the orientation of $P$. By Claim 8, we obtain that there exists a component $H'$ of $G-P$ such that $H'$ connects $H_{i_1}^{j_{1}}$ and $H_{i_2}^{j_{2}}$. Therefore, there exists at least one edge between $H'$ and $H_{i_1}^{j_{1}}$, and at least one edge between $H'$ and $H_{i_2}^{j_{2}}$. Let $w\in V(H_{i_{1}}^{j_{1}})$ and $z\in V(H_{i_{2}}^{j_{2}})$ such that $N(w)\cap V(H')\ne \emptyset $ and $N(z)\cap V(H')\ne \emptyset $. And let $P_{wz}$ be a shortest $(w,z)$-path whose all internal vertices belong to $H'$. Without loss of generality, assume that $P_{wz}=wz_{1}z_{2}... z_{s}z$. By the shortest path property of $P_{wz}$, we obtain that $wz_{1}z_{2}$ is an induced $3$-path. By the definition of $1$-clique, we have $|N(w)\cap H_{P}^{+}|\le 1$ and $|N(z)\cap H_{P}^{+}|\le 1$. And we can conclude that $|(N(z_{1})\cup N(z_{2}))\cap H_{P}^{+}|\ge h-k+\ell-1\ge k\ge 2$. Otherwise, $|(N(z_{1})\cup N(z_{2}))\cap H_{P}^{+}|\le h-k+\ell-2$. Then $|H_{P}^{+}\setminus (((N(z_{1})\cup N(z_{2}))\cap H_{P}^{+})\cup (N(w)\cap H_{P}^{+}))|\ge h-(h-k+\ell-2)-1\ge k $. Let $x_{i_{1}}^{+},x_{i_{2}}^{+},...,x_{i_{k}}^{+}\in H_{P}^{+}\setminus (((N(z_{1})\cup N(z_{2}))\cap H_{P}^{+})\cup (N(w)\cap H_{P}^{+}))$ be $k$ distinct vertices. Therefore, $\{w,z_{1},z_{2}\}\cup \{x_{i_{1}}^{+},x_{i_{2}}^{+},...,x_{i_{k}}^{+}\}$ induces a $P_{3}\cup kP_{1}$, a contradiction. Let $x_{a}^{+},x_{b}^{+}\in (N(z_{1})\cup N(z_{2}))\cap H_{P}^{+}$ be two distinct vertices. Without loss of generality, assume $a<b$. Let $y_{a},y_{b}\in V(H)$ with $x_{a}y_{a},x_{b}y_{b}\in E(G)$, and let $P_{ab}$ be a $(y_{a},y_{b})$-path in $H$. Since $|N(w)\cap H_{P}^{+}|\le 1 $, $|N(z)\cap H_{P}^{+}|\le 1$, $\{w ,z\}\subseteq N_{P}(H')$ and $H_{P}\subseteq V_{h-k+\ell}$, we have $H\ne H'$. Hence, there exists an $(x_{a}^{+},x_{b}^{+})$-path $P_{H'}$ whose all internal vertex belong to $H'$ and are disjoint from $V(P_{ab})$. Therefore, $u\overrightarrow{P}x_{a}y_{a}\allowbreak P_{ab}y_{b}x_{b}\overleftarrow{P}x_{a}^{+}P_{H'}x_{b}^{+}\overrightarrow{P}v$ is a longer $(u,v)$-path than $P$, a contradiction.
\end{proof}

	If $u$ and $v$ are contained in two distinct maximal $1$-cliques of $\mathcal{H}$, then $|\mathcal{H}|=|V_{h-k+\ell}|+1$. By Claim 9, $u$ and $v$ may be contained in the same component of $G-V_{h-k+\ell}$. Then $c(G-V_{h-k+\ell})\ge |\mathcal{H}|-1+1=|\mathcal{H}|$. Therefore, $\tau (G)\le \frac{|V_{h-k+\ell}|}{c(G-V_{h-k+\ell})}\le \frac{|\mathcal{H}|-1}{|\mathcal{H}|}<1$, a contradiction.
	
	If one of $u$ and $v$ is contained in a $1$-clique, then $|\mathcal{H}|=|V_{h-k+\ell}|$. By Claim 9, any two maximal $1$-cliques of $\mathcal{H}$ are contained in distinct components of $G-V_{h-k+\ell}$. Then $c(G-V_{h-k+\ell})\ge |\mathcal{H}|+1$. Therefore, $\tau (G)\le \frac{|V_{h-k+\ell}|}{c(G-V_{h-k+\ell})}\le \frac{|\mathcal{H}|}{|\mathcal{H}|+1}<1$, a contradiction.
	
	If both $u$ and $v$ are not contained in any $1$-cliques, i.e., $u$ and $v$ are $(h-k+\ell)$-vertices, then $|\mathcal{H}|+1=|V_{h-k+\ell}|$. In this situation, any $1$-clique does not contain $u$ or $v$. By Claim 9, any two maximal $1$-cliques of $\mathcal{H}$ are contained in distinct components of $G-V_{h-k+\ell}$. Then $c(G-V_{h-k+\ell})\ge |\mathcal{H}|+1$. Therefore, $\tau (G)\le \frac{|V_{h-k+\ell}|}{c(G-V_{h-k+\ell})}\le \frac{|\mathcal{H}|+1}{|\mathcal{H}|+1}=1$, a contradiction.$\hfill\Box$

\end{document}